\documentclass{amsart}

\newtheorem{theorem}{Theorem}[section]
\newtheorem{lemma}[theorem]{Lemma}
\newtheorem{proposition}{Proposition}[section]
\usepackage{amsmath,amssymb,amsthm}
\theoremstyle{definition}
\newtheorem{definition}[theorem]{Definition}

\newtheorem{corollary}[theorem]{Corollary}
\theoremstyle{remark}
\newtheorem{remark}[theorem]{Remark}
\usepackage{enumitem}
\numberwithin{equation}{section}
\begin{document}
	
	\title{Linear isometries on Weighted Coordinates Poset Block Space}
	
	\author{Atul Kumar Shriwastva}
	\address{{Department of Mathematics, National Institute of Technology Warangal, Hanamkonda, Telangana 506004, India}}
	\email{shriwastvaatul@student.nitw.ac.in}
	\thanks{}
	
	\author{R. S. Selvaraj}
	\address{{Department of Mathematics, National Institute of Technology Warangal, Hanamkonda, Telangana 506004, India}}
	\email{rsselva@nitw.ac.in}
	\date{\today}
	\begin{abstract}
	 Given $[n]=\{1,2,\ldots,n\}$, a poset order $\preceq$ on $[n]$, a label map $\pi : [n] \rightarrow \mathbb{N}$ defined by $\pi(i)=k_i$ with $\sum_{i=1}^{n}\pi (i) = N$, and a weight function $w$ on $\mathbb{F}_{q}$, let  $\mathbb{F}_{q}^N$ be the vector space of $N$-tuples over the field $\mathbb{F}_{q}$ equipped with $(P,w,\pi)$-metric where $ \mathbb{F}_q^N $ is the direct sum of spaces $ \mathbb{F}_{q}^{k_1}, \mathbb{F}_{q}^{k_2}, \ldots, \mathbb{F}_{q}^{k_n} $. In this paper, we determine the groups of linear isometries of $(P,w,\pi)$-metric spaces in terms of a semi-direct product, which turns out to be similar to the case of poset (block) metric spaces. In particular, we re-obtain the group of linear isometries of the $(P,w)$-mertic spaces and $(P,\pi)$-mertic spaces.		
	\end{abstract}
\maketitle
\section{Introduction}
 Let $[n] = \{1,2,\ldots,n\}$ represents the coordinate positions of $n$-tuples  in the vector space $\mathbb{F}_q^n$.
 Brualdi et al. introduced poset metric \cite{Bru} on $\mathbb{F}_q^n$ by using partially ordered relation on $[n]$. Motivated by Brualdi et al.,  K. Feng \cite{fxh} introduced a metric known as $\pi$-metric on $\mathbb{F}_q^N$ by using a label map $ \pi : [n] \rightarrow \mathbb{N}$ such that $\sum_{i=1}^{n}\pi (i) = N$ and $\mathbb{F}_{q}^{N} \equiv \mathbb{F}_{q}^{\pi(1)}  \oplus \mathbb{F}_{q}^{\pi(2)} \oplus \ldots \oplus \mathbb{F}_{q}^{\pi(n)}$.  Thus, metrics on $\mathbb{F}_q^N$ become a new research for researchers to explore it. Errors within $\leq \frac{d_{\pi}(\mathbb{C})-1}{2}$ blocks may be corrected using a code $\mathbb{C}$ with $\pi$-metrics (linear error-block codes) where $d_{\pi}(\mathbb{C})$ is the minimum distance of $\mathbb{C}$. The creation of cryptographic schemes can also be done using block codes with different metrics. Block codes have several applications in experimental design, high-dimensional numerical integration, and cryptography.
Further,  Alves et al. \cite{Ebc}, introduced  $(P,\pi)$-metric on $\mathbb{F}_q^N$ with the help of partial order on the block positions $[n]$. I. G. Sudha and R. S. Selvaraj introduced
pomset mteric \cite{gsrs} on $\mathbb{Z}_m^n$ with the help of multiset concept and partial order relation on the multiset which is a generalization of Lee space \cite{Leecode}, in particular, and poset space, in general, over $\mathbb{Z}_m$. However, L. Panek \cite{wcps} introduced the  weighted coordinates poset metric recently (2020) which is a simplified version of the pomset metric that does not use the multiset structure. 
\par In \cite{as}, we defined the weighted coordinates poset block metric ($d_{(P,w,\pi)}$) on the space $\mathbb{F}_q^N$. It extends the weighted coordinates poset metric ($(P,w)$-metric) \cite{wcps} introduced by L. Panek and J. A. Pinheiro and generalizes the poset block metric ($(P,\pi)$-metric) \cite{Ebc} introduced by  M. M. S. Alves et al..
Before defining the weighted coordinates poset block metric on $\mathbb{F}_q^N$, we will recall certain basic definitions in order to facilitate the organization of this paper. If $R$ is a ring and $N$ is a positive integer, a map $w : R^N \rightarrow \mathbb{N} \cup \{0\}$ is said to be a weight on $R^N $ if it satisfies the following properties: 
$(a)$  $w(u) \geq 0$;  $u \in R^N $
$(b)$  $w(u) = 0$ iff  $ u = 0$ 
$(c)$  $w(- u)  = w(u) $;  $u \in R^N $ 
$(d)$  $w(u + v)  \leq w(u) + w(v)  $; $u,v \in R^N $.
\par
Let $P=([n],\preceq)$ be a poset.  An element $j \in J \subseteq P$ is said to be a maximal element of $J$ if there is no $i \in J$ such that $j \preceq i$. An element $j \in J \subseteq P$ is said to be a minimal element of $J$ if there is no $i \in J$ such that $i \preceq j$.  A subset $I$ of $P$ is said to be an ideal if $j \in I$ and $i \preceq j$ imply $i \in I$. For a subset $J$  of $P$, an ideal generated by $J$ is the smallest ideal containing $J$ and is denoted by $\langle J \rangle$.  
\par Let  $w$ be a weight on $\mathbb{F}_q$ and $M_w = \max \{w(\alpha) : \alpha \in \mathbb{F}_q \}$. For a $k \in \mathbb{N}$, and a $v=(v_1, v_2,\ldots,v_k) \in  \mathbb{F}_q^{k} $, we define  $\tilde{w}^{k}{(v)} = \max \{ w(v_i) : 1 \leq i \leq k\}$. Clearly, $\tilde{w}^{k}$ is a weight on $\mathbb{F}_q^{k} $ induced by the weight $w$. On 	$ \mathbb{F}_{q}^{k_i} $, $1 \leq i \leq n$, we call $\tilde{w}^{k_i}$, a block weight.
\begin{definition}
  Given a partial order $\preceq $ on $[n]=\{1,2,\ldots,n\}$, the pair $P=([n],\preceq)$ is a poset.  With a label map  $\pi  : [n] \rightarrow \mathbb{N} $ defined as $\pi (i) =k_i $
  in the previous page such that $ \sum\limits_{i=1}^{n} \pi (i) = N  $, a positive inetger, we have $	\mathbb{F}_{q}^{N} \equiv \mathbb{F}_{q}^{k_1}  \oplus \mathbb{F}_{q}^{k_2} \oplus \ldots \oplus \mathbb{F}_{q}^{k_n} $. Thus, if  $x \in  \mathbb{F}_{q}^{N} $ then $x= x_1 \oplus x_2 \oplus \cdots \oplus x_n $ with  $x_i = (x_{i_1},x_{i_2},\ldots,x_{i_{k_i}}) \in \mathbb{F}_{q}^{k_i} $. Let $I_{x}^{P,\pi} =  \langle supp_{\pi}(x) \rangle $ be the ideal generated by the $\pi$-support of $x$ and $M_{x}^{P,\pi}$ be the set of all maximal elements in $I_{x}^{P,\pi}$. The weighted coordinates poset block weight or $(P,w,\pi)$-weight of  $x \in \mathbb{F}_q^N$ is defined as
\begin{equation*}
	w_{(P,w,\pi)}(x) \triangleq \sum\limits_{i \in M_{x}^{P,\pi}} \tilde{w}^{k_i}{(x_i)} + \sum\limits_{i \in {I_{x}^{P,\pi} \setminus M_{x}^{P,\pi}}} M_w
\end{equation*} 
The  $(P,w,\pi)$-distance between two vectors  $x,y \in \mathbb{F}_q ^N$ is defined as:
$ d_{(P,w,\pi)}(x,y) \triangleq  w_{(P,w,\pi)}(x-y)$.
$d_{(P,w,\pi)}$ defines a metric on $ \mathbb{F}_{q}^N $ called as \textit{weighted cordinates poset block metric} or $ (P,w,\pi) $-metric. The pair $ (\mathbb{F}_{q}^N,~d_{(P,w,\pi)} )$ is said to be a  $(P,w,\pi)$-space.
\end{definition}
 A $ (P,w,\pi) $-block code  $\mathbb{C} $ of length $N$ is a subset of $ (\mathbb{F}_{q}^N,~d_{(P,w,\pi)} )$-space  and 
$ d_{(P,w,\pi)}\mathbb{(C)} = min \{  d_{(P, w,\pi)} {(c_1, c_2)}: c_1,~ c_2 \in \mathbb{C} \} $ gives the minimum distance of  $\mathbb{C}$. 
If $\mathbb{C}$ is a linear $ (P,w,\pi) $-block code, then 
$ d_{(P,w,\pi)}\mathbb{(C)} = min \{ w_{(P,w,\pi)}(c) : 0 \neq c \in \mathbb{C} \} $.
It  is clear that  $ w_{(P,w,\pi)}(v) \leq n M_w$ for any $v \in \mathbb{F}_{q}^N $. Thus, the minimum distance of  $\mathbb{C} $ is bounded above by $ n M_w $.
\begin{itemize}  
	\item 	If $w$ is the Hamming weight  on $\mathbb{F}_q$, then the  $(P,w,\pi)$-space becomes the $(P,\pi)$-space  (as in \cite{Ebc}).
	\item 	If $k_i = 1 $ for every $i\in [n]$ and $w$ is the Hamming weight on $\mathbb{F}_q$,  then  the	 $(P,w,\pi)$-space becomes the poset space or $P$-space (as in \cite{Bru}).		
	\item 	If $w$ is the Hamming weight on $\mathbb{F}_q$ and $P$ is an antichain, then  the  $(P,w,\pi)$-space becomes the $\pi$-space or $( \mathbb{F}_q^N,~d_{\pi})$-space  (as in \cite{fxh}). 
	\item If $k_i = 1 $ for every $i\in [n]$ then the $(P,w,\pi)$-space becomes the $(P,w)$-space (as in  \cite{wcps}).
\end{itemize} 
\par Now, we start with defining basic thing about linear isometry on $\mathbb{F}_{q}^N$ and then proceed on determining the groups of linear isometries of $(P,w,\pi)$-metric spaces.
\par A linear isometry $T$ of the metric space $ (\mathbb{F}_{q}^N,~d_{(P,w,\pi)} )$  is a linear transformation $T : \mathbb{F}_{q}^N \rightarrow  \mathbb{F}_{q}^N $ which preserves $(P,w,\pi)$-distance. That is  $d_{(P,w,\pi)} (T(x) , T(y)) =  d_{(P,w,\pi)} (x,y) $ for every $x,y \in  \mathbb{F}_{q}^N  $. In other way, a linear transformation $T : \mathbb{F}_{q}^N \rightarrow  \mathbb{F}_{q}^N $ is said to be an isometry if  $w_{(P,w,\pi)} (T(x)) =  w_{(P,w,\pi)} (x) $ for every $x\in  \mathbb{F}_{q}^N  $. A linear isometry of  $ (\mathbb{F}_{q}^N,~d_{(P,w,\pi)} )$ is said to be a $(P,w,\pi)$-isometry. 
Set of all linear isometries of  $ (\mathbb{F}_{q}^N,~d_{(P,w,\pi)} )$ forms a group, called as group of linear isometry of  $ (\mathbb{F}_{q}^N,~d_{(P,w,\pi)} )$  and denoted by ${{LI}som}_{(P,w,\pi)} (\mathbb{F}_{q}^N) $.

\par Linear isometries take linear codes onto linear with preserving their length, dimension, minimum distance, and other parameters, so it is used to classify linear codes in equivalence classes. Therefore, if one of two linear codes is the other's mirror image under a linear isometry, it is only appropriate to refer to them as equivalent codes.
The study of full description of linear symmetries in particular cases (with label $ \pi (i)= 1 $ $\forall$ $ i \in [n] $) of poset spaces such as Rosenbloom-Tsfasman spaces, crown spaces, and weak spaces were determined by the authors K. Lee \cite{Automorphism group of the Rosenbloom-Tsfasman space}, S. H. Cho and D. S. Kim \cite{Automorphism group of the	crown-weight space}, and D. S. Kim \cite{weak spaces LISo}, respectively. Inspired by them,
L. Panek, M. Firer, H. K. Kim, and J. Y. Hyun \cite{Lmhj} provided a comprehensive description of the groups of linear symmetries in those spaces with label $ \pi (i)= 1 $ $\forall$ $ i \in [n] $.
\par After that, researchers are interested in determining the isometry group of a poset-metric space, which need not be linear. The full symmetry group (which includes non-linear isometries) of arbitrary poset space and a particular case of poset spaces that are product of Rosenbloom-Tsfasman spaces are described by J. Y. Hyun \cite{A subgroup of the full poset-isometry group}, and L. Panek et al.  \cite{Symmetry groups of RT spaces}, respectively.  In \cite{posetadmitLISO}, the authors characterize the posets that admit the linearity of isometries.
\par
The group of full linear isometries of $(P,\pi)$-metric spaces and $\pi$-metric spaces with label $ \pi (i)= 1 $ $\forall$ $ i \in [n] $ were described by M. M. S. Alves in \cite{Ebc}. Recently, L. Panek et al. \cite{wcps} approached the similar way as in \cite{Lmhj} to determine linear isometry of $ (P,w)$-metric spaces with label $ \pi (i)= 1 $ $\forall$ $ i \in [n] $ and got a similar result as described in \cite{Lmhj}. In this work, we find linear
isometries of $ (P,w,\pi)$-metric spaces with any given label $ \pi (i)= k_i $ $\forall$ $ i \in [n] $, a weight $w$ on $\mathbb{F}_{q}$, and poset $P$.
\par 
 We begin with initially as same concept in \cite{Lmhj}, to associate to each isometry $T$ an automorphism $\psi_T$ of the underlying poset $P$ (Theorem \ref{structureTheorem}). We choose a more coordinate-free methodology, and the block's dimensions introduce a new constraint. These are the primary distinctions. 
The main difference relies on the fact we are considering a general weight $ w $ instead of the Hamming weights (or Lee weights) on $ \mathbb{F}_q$ and one additional weight $ \tilde{w} $ (depends on $w$) on $ \mathbb{F}_q^{k_i}$ for each label $i \in [n] $.
 We find two subgroups of isometries: one induced by automorphisms of $P$ that preserve labels and the other by the identity map on $P$. 
Finally, we prove some results on linear isometries similar to the ones found in \cite{Lmhj}, and \cite{Ebc}, and conclude that ${{LI}som}_{(P,w,\pi)} (\mathbb{F}_{q}^N)$ is the semi-direct product of those two subgroups.
\section{Subgroups of a group of Linear Isometries}
Let $\mathcal{B}_{j} = \{ e_{{j},z} :  1 \leq z \leq k_{j} \}$ be the canonical basis of $\mathbb{F}_{q}^ {k_j}$ for each $ {j} \in [n]$ and 
$\mathcal{B} = \{ e_{j, {z}} : 1 \leq {j} \leq n,~  e_{j, z} \in\mathcal{B}_{j} \}$ be a basis for  $ \mathbb{F}_{q}^N$. 
A bijection map $\gamma : P \rightarrow P$ is said to be an order automorphism if $\gamma$ and $\gamma^{-1}$ preserves the order relation of $P$. Let $\mathcal{AUT}(P)$ denote the group of order automorphisms of given a poset $(P=([n], \preceq ))$. Let $\pi  : [n] \rightarrow \mathbb{N} $ be a label map of the poset $P$ such that  $\pi ({j}) =k_{j} >0$ for each ${j} \in [n] $. The subgroup of automorphisms $\psi \in \mathcal{AUT}(P)$ such that $ k_{\psi ({j})} = \pi (\psi ({j})) = \pi ({j}) = k_{j}$ for all ${j} \in [n] $ is denoted by $\mathcal{AUT}(P,\pi)$ and is called the group of automorphisms of $(P,\pi)$ which preserve  labels.
\par  The linear mapping $T_{\psi} : \mathbb{F}_{q}^N \rightarrow  \mathbb{F}_{q}^N $ such that $ T_{\psi} (e_{j,z}) = e_{\psi(j),z}$,  associates  each $\psi \in \mathcal{AUT}(P)$ to the  $ T_{\psi}$. Since definition of $T_{\psi}$ only makes sense if $dim( \mathbb{F}_{q}^{k_i}) = dim( \mathbb{F}_{q}^{k_j})$.
\par Let  the map  $ \Gamma: \mathcal{AUT}(P, \pi)  \rightarrow {LIsom}_{(P,w,\pi)} (\mathbb{F}_{q}^N) $ defined by $ \psi \rightarrow T_{\psi} $. Let $ \beta, \delta \in \mathcal{AUT}(P, \pi)  $ then
$  T_{\beta \delta} (e_{j,z}) = e_{ (\beta \delta) ({j}),z } = T_{\beta }( e_{ \delta (j) ,z}) = T_{\beta } T _{ \delta } ( e_{{j,z} })  $. 
Thus, $\Gamma$ is trivially a homomorphism and injective (injectivity follows from the definition of $\Gamma$). $\mathcal{I}mg(\Gamma)$ denote the  image of $\Gamma$ which is  a subgroup of  ${LIsom}_{(P,w,\pi)} (\mathbb{F}_{q}^N)$  and  isomorphic to $\mathcal{AUT}(P, \pi) $. And, $T_{\psi} (\mathbb{F}_{q}^{k_{j}}) =\mathbb{F}_{q}^{k_{\psi(j)}}$. 
  \begin{proposition}
 If  $\psi \in \mathcal{AUT}(P, \pi)$  then the linear mapping $T_{\psi}$ is a linear isometry of $ (\mathbb{F}_{q}^N,~d_{(P,w,\pi)} )$.
  \end{proposition}
  \begin{proof}
  	 Let $x = \sum\limits_{j , z} \eta _{j z} e_{j,z} \in \mathbb{F}_{q}^N $, then we get  
  	 \begin{align*}
  	 	I_{T_{\psi} (x)}^{P,\pi}  &= \langle supp_{\pi} (T_{\psi} (x)) \rangle \\ &= \langle supp_{\pi} \big( \sum\limits_{j,z} \eta _{j z} e_{\psi(j),z}  \big) \rangle \\& =\langle \{\psi{(j)} \in P : \eta _{j z} \neq 0 ~\text{for some } z \} \rangle \\ 
  	 	& = \langle \{\psi{(j)} \in P :  j \in supp_ {\pi} (x)\} \rangle 
  	 	\\ &= {\psi} (\langle supp_{\pi}  (x) \rangle) \\&= \psi(I _{x}^{P,\pi}).
  	 \end{align*}
   Since $\psi$ is an oder automorphism of $P$ then 
   $ \psi(M _{x}^{P,\pi}) = M_{T_{\psi} (x)}^{P,\pi} $. So,  $ \psi( I_{x}^{P,\pi} \setminus M _{x}^{P,\pi}) = I_{T_{\psi} (x)}^{P,\pi} \setminus  M_{T_{\psi} (x)}^{P,\pi}  $. Thus,
  \begin{align*}
  	w_{(P,w,\pi)} (T_{\psi}(x)) &=  \sum\limits_{ j \in M_{T_{\psi} (x)}^{P,\pi}}  \tilde{w}{(x_{\psi^{-1}(j)})} + \sum\limits_{j \in {I_{T_{\psi} (x)}^{P,\pi}} \setminus{ M_{T_{\psi} (x)}^{P,\pi}} } M_w  \\
  	&= \sum\limits_{ j \in {\psi( M _{x}^{P,\pi}} ) }  \tilde{w}{(x_{\psi^{-1}(j)})} + \sum\limits_{ j \in {\psi  (I _{x}^{P,\pi} \setminus  M _{x}^{P,\pi} }) } M_w \\
  		&= \sum\limits_{\jmath \in {M _{x}^{P,\pi}}  }  \tilde{w}{(x_{\jmath})} + \sum\limits_{\jmath \in {I _{x}^{P,\pi} \setminus  M _{x}^{P,\pi} } } M_w  \\
  		&=	w_{(P,w,\pi)} (x).
  \end{align*}
  Hence $T_{\psi}$ preserves $(P,w,\pi)$-weights. 
  \end{proof}
 Given an $X \subseteq P$, we define ${ (\mathbb{F}_{q}^N)}_X$ to be the subspace ${ (\mathbb{F}_{q}^N)}_X = \{ v \in  \mathbb{F}_{q}^N : supp_{\pi}  (v) \subseteq X\}$. In particular, if $ \tilde{w}( \gamma_{j z} ) =  \tilde{w} ( 1 )$ then $ \tilde{w}(\alpha_j \gamma_{j z} ) =  \tilde{w} ( \alpha_j )$ $\forall$ $\alpha_j \in  \mathbb{F}_{q}^{k_j}$. But if we consider $\alpha_j \in  \mathbb{Z}_{m}^{k_j}$ in place of $\alpha_j \in  \mathbb{F}_{q}^{k_j}$ then it  need not be true because it  contains zero divisors. 
\begin{proposition}\label{T collection}
	Let  $T : \mathbb{F}_{q}^N \rightarrow  \mathbb{F}_{q}^N $ be a linear isomorphism  such that for each $j \in [n]$, 
	\begin{equation*}
		T (e_{j,z}) = \gamma_{j z}  e_{ j,z } + v^{j}
	\end{equation*}
	where $  v^{j} \in (\mathbb{F}_{q}^N)_{{\langle {j} \rangle} ^{*}} $, $ \tilde{w}( \gamma_{j z} ) =  \tilde{w} ( 1 )$, and $ \tilde{w}(\alpha_j \gamma_{j z} ) =  \tilde{w} ( \alpha_j )$ $\forall$ $\alpha_j \in  \mathbb{F}_{q}^{k_j}$. Then
	$T$ is a linear isometry of  $ (\mathbb{F}_{q}^N , ~d_{( P,w,\pi)} )$.
\end{proposition}
\begin{proof}
	Since $			T (e_{j,z}) = \gamma_{j z}  e_{ j,z } + v^{j} $,  where $  v^{j} \in (\mathbb{F}_{q}^N)_{{\langle j \rangle} ^{*}}$  and $\tilde{w}(\alpha_j \gamma_{j z} ) =  \tilde{w} ( \alpha_j )$ $\forall$ $\alpha_j \in  \mathbb{F}_{q}^{k_j}$. If $ x = \sum\limits_{j,z} \Theta_{j z} e_{ j,z } $ then,
	\begin{equation*}
		T(x) = \sum\limits_{j,z} \Theta_{j z} T (e_{ j,z })= \sum\limits_{j,z} \eta_{j z}  e_{j,z} + \delta^{j}
	\end{equation*}
	where $  \eta_{j z} = \Theta_{j z} \gamma_{ j,z }  $, $\delta^{j} = \Theta_{j,z}  v^{j} \in (\mathbb{F}_{q}^N)_{{\langle j \rangle} ^{*}} $ and $ \tilde{w}( \eta_{j z} ) =  \tilde{w} ( \Theta_{j z} ) $ with $ \eta_{j z}  \neq 0 $ for all $j$ such that $ \Theta_{j z} \neq 0 $. Clearly, $supp_{\pi}(x) \subseteq supp_{\pi}{( T(x))}$
	\par
	Let $\delta^j = \delta_{1}^j + \delta_{2}^j + \ldots + \delta_{n}^j = \sum\limits_{i,z} \delta_{i z}^j e_{i, z}$ be the the canonical decomposition of $\delta^j$ in $ \mathbb{F}_{q}^{N}$.  Note that if $ \delta_{i z}^j \neq 0 $ means $ \delta_{i }^j\neq 0 $ then $i \prec_P j $ because  $\delta^{j} \in (\mathbb{F}_{q}^N)_{{\langle j \rangle} ^{*}} $. 
	\par If $ i \in M_x^{P,\pi}$ then all $ \delta_{i z}^k  $  are zero for each $k$, because if $ \delta_{i z}^k \neq 0 $
	then  $ \eta_{k z} \neq 0 $ and hence $ \Theta_{k z} \neq 0 $. Therefore $k \in supp_{\pi} (x) $ and $i \prec_P k$, but $i$ is maximal in $ supp_{\pi}  (x) $. 
	\par $T(x)$ can be written as 
	\begin{align*}
		T(x) &= \sum\limits_{j,z} (\eta_{j z}  e_{j,z} + (\sum\limits_{i,z} \delta_{i z}^j e_{i, z})) \\&= \sum\limits_{j,z} (\eta_{j z}  e_{j,z} + ( \delta_{1 z}^j e_{1, z} + \delta_{2 z}^j e_{2, z} + \cdots +\delta_{n z}^j e_{n, z})) \\&=
		\sum\limits_{\jmath,z} ( \eta_{\jmath z } + (\delta_{\jmath  z}^1 + \delta_{\jmath  z}^2 + \cdots + \delta_{\jmath z}^n )  )e_{\jmath,  z }
	\end{align*}
	Suppose that $ j \in M_x^{P,\pi}$ and  $ j \notin supp_{\pi}(T(x) )$ then $j^{th}$ term of $T(x)$, $
	\eta_{j z } + (\delta_{j  z}^1 + \delta_{j  z}^2 + \cdots + \delta_{j z}^n )  = 0 $.
	Since $ \delta_{j  z}^k = 0 $ for each $k$ so, $	\eta_{j  z} =0 $, a contradiction. Therefore  $ j \in supp_{\pi}(T(x)) $ and  $  M_{x}^{P,\pi} \subseteq supp_{\pi}(T(x)) $.
	\par Suppose the $i^{th}$ label of $ T (x) $  
	\begin{equation*}
		\eta_{i  z} + (\delta_{i  z}^1 + \delta_{i  z}^2 + \cdots + \delta_{i  z}^n )  
	\end{equation*}
	is maximal, $ i \in M_{T(x)}^{P,\pi}$. If $ \delta_{i z}^k \neq 0 $ then $k \in supp_{\pi}(x) $ and $  i \prec_P k \prec_P j$ for some $ j \in M_{x}^{P,\pi}  \subseteq supp_{\pi}(T(x)) $
	which implies $i$ is not maximal, a contradiction. Hence all $ \delta_{i z}^k = 0 $ for each $k$ and since $ \eta_{i z} \notin 0 $, we have that $i \in supp_{\pi}(x)$. If $ i \notin M_{x}^{P,\pi}$ then $  i  \prec_P j$ for some  
	$ j \in M_{x}^{P,\pi}  \subseteq supp_{\pi}(T(x)) $, which implies $ i \notin M_{T(x)}^{P,\pi} $, again a contradiction. Hence  $ i \in M_{x}^{P,\pi} $ and it follows that $  M_{T(x)}^{P,\pi} \subseteq M_{x}^{P,\pi}$. 
	\par Since $  M_{x}^{P,\pi} \subseteq supp_{\pi}(T(x)) $,  $  M_{T(x)}^{P,\pi}  \subseteq  M_{x}^{P,\pi}$ and $ \tilde{w}( \eta_{i z}  ) =  \tilde{w} ( \Theta_{i z}) $ for all $i$, thus
	$	w_{(P,w,\pi)} (x) = 	w_{(P,w,\pi)} (T(x))$. Therefore  $T$ is a linear isometry of  $ (\mathbb{F}_{q}^N , ~d_{( P,w,\pi)} )$.
\end{proof}
Let $\mathcal{T}$ be the set of all mapping defined in the previous Proposition \ref{T collection}. We will prove in  Theorem \ref{structureTheorem} that $\mathcal{T}$ is a subgroup of ${LIsom}_{(P,w,\pi)} (\mathbb{F}_{q}^N)$. We can also obtain a matrical version of this group.
\par Now, let $B = ( B_{{i_1}}, B_{{i_2}}, \ldots,B_{{i_n}} ) $ be a total ordering of the basis of  $ \mathbb{F}_{q}^N$ such that $B_{i_s}$ appears before $B_{i_r}$ whenever 
$  	w_{(P,w,\pi)}(e_{{i_s},j}) < w_{(P,w,\pi)}(e_{{i_r},j} ) $ for all  $ {i_r}, i_s = 1,2,\ldots , n $. Renaming the elements of $ P = ([n], \preceq)$ if necessary, we can  suppose that  $ {i_r} =r$ for all $r = 1,2,\ldots , n $. In this manner, $B = ( B_{1}, B_{2}, \ldots, B_{n} ) $
and if  
$  	w_{(P,w,\pi)}(e_{s,j} ) < w_{(P,w,\pi)}(e_{r,j} ) $  then  all elements of  $ B_{s}$ come before the elements of $B_r$ and $s \prec_P r $ or $s \preceq_P r $.
\begin{theorem}\label{matric version}
	Let $B = ( B_{1}, B_{2}, \ldots, B_{n} ) $ be the canonical basis of $\mathbb{F}_{q}^ {N}$ where  $  	w_{(P,w,\pi)}(e_{i,z} ) \leq w_{(P,w,\pi)}(e_{\jmath,z} ) $ implies $ i \preceq_P \jmath$.
	If $T \in \mathcal{T}$ then
	\begin{equation*}
		T(e_{i,z}) = \sum\limits_{i \preceq_P \jmath}\sum\limits_{t=1}^{k_i}  \eta_{i t}^{\jmath z} e_{i,t}
	\end{equation*}
	where each block $\bigg(\eta_{rt}^{rz} \bigg)_{1 \leq t \leq k_r} ^{1 \leq z \leq k_r}$, $r = 1,2,\ldots,n$, is an invertible matrix with $ \tilde{w}(\eta_{rt}^{rz}) =\tilde{w}(1) $ and $  \tilde{w}( \alpha \eta_{rt}^{rz}) =\tilde{w}(\alpha) $ for all $r \in [n]  $ and $\alpha \in \mathbb{F}_{q}^ {k_r}$. Every element of $ \mathcal{T}$ is represented as an upper-triangular matrix with respect to $ B $.
\end{theorem}
  \begin{proof}
  	Since $T \in \mathcal{T}$ we have that $ T( \mathbb{F}_{q}^{k_i} ) \subseteq (\mathbb{F}_{q}^N)_{{\langle i \rangle} ^{*}} $. So
  	\begin{equation*}
  		\begin{split}
  			T( e_{1,1} ) &= \eta_{11}^{11} e_{1,1} + \eta_{12}^{11} e_{1,2} + \cdots + \eta_{1{k_1}}^{11} e_{1,{k_1}} \\
  			T( e_{1,2} ) &= \eta_{11}^{12} e_{1,1} + \eta_{12}^{12} e_{1,2} + \cdots + \eta_{1{k_1}}^{12} e_{1,{k_1}} \\
  			&\vdots \\
  				T( e_{1,{k_1}} ) &= \eta_{11}^{1{k_1}} e_{1,1} + \eta_{12}^{1{k_1}} e_{1,2} + \cdots + \eta_{1{k_1}}^{1{k_1}} e_{1,{k_1}} 
  				\\
  	T( e_{2,1} ) &= (\eta_{11}^{21} e_{1,1} + \eta_{12}^{21} e_{1,2} +\cdots+\eta_{1{k_1}}^{21} e_{1,{k_1}}) + \\ & \hspace{5.2cm} (\eta_{21}^{21} e_{2,1} + \eta_{22}^{21} e_{2,2} +\cdots+\eta_{2{k_2}}^{21} e_{2,{k_2}})  \\
  	T( e_{2,2} ) &= (\eta_{11}^{22} e_{1,1} + \eta_{12}^{22} e_{1,2} + \cdots + \eta_{1{k_1}}^{22}e_{1,{k_1}})  + \\ & \hspace{5.2cm}  (\eta_{21}^{22} e_{2,1} + \eta_{22}^{22} e_{2,2} + \cdots + \eta_{2{k_2}}^{22}e_{2,{k_2}}) \\
  	&\vdots \\
  	T( e_{2,{k_2}} ) &= (\eta_{11}^{2{k_2}} e_{1,1} + \eta_{12}^{2{k_2}} e_{1,2} + \cdots + \eta_{1{k_1}}^{2{k_2}}e_{1,{k_1}})  + \\ & \hspace{5.2cm} (\eta_{21}^{2{k_2}} e_{2,1} + \eta_{2{k_2}}^{2{k_2}} e_{2,2} + \cdots + \eta_{2{k_2}}^{2{k_2}}e_{2,{k_2}})  \\
  	&\vdots  	  \\  
  	T( e_{n,1} ) &= (\eta_{11}^{n1} e_{1,1} + \eta_{12}^{n1} e_{1,2} +\cdots+\eta_{1{k_1}}^{n1} e_{1,{k_1}}) + \cdots + \\ & \hspace{4.8cm} (\eta_{n1}^{n1} e_{n,1} + \eta_{n2}^{n1} e_{n,2} +\cdots+\eta_{n{k_n}}^{n1} e_{n,{k_n}}) \\
  	T( e_{n,2} ) &= (\eta_{11}^{n2} e_{1,1} + \eta_{12}^{n2} e_{1,2} + \cdots + \eta_{1{k_1}}^{n2} e_{1,{k_1}})  +\cdots+ \\ & \hspace{5.2cm} (\eta_{21}^{n2} e_{2,1} + \eta_{22}^{n2} e_{2,2} + \cdots + \eta_{2{k_2}}^{n2} e_{2,{k_2}}) \\
  	&\vdots \\
  	T( e_{n,{k_n}} ) &= (\eta_{11}^{n{k_n}} e_{n,1} + \eta_{12}^{n{k_n}} e_{n,2} + \cdots + \eta_{1{k_1}}^{n{k_n}} e_{n,{k_1}}) +\cdots+ \\ & \hspace{5.2cm}  (\eta_{n1}^{n{k_n}} e_{n,1} + \eta_{n{2}}^{n{k_n}} e_{n,2} + \cdots + \eta_{n{k_n}}^{n{k_n}} e_{n,{k_n}}) 
  \end{split}
\end{equation*}
  	where $ ( \eta_{s1}^{ij}, \eta_{s2}^{ij}, \dots, \eta_{s{k_s}}^{ij} )  = 0 $  if $ s \nleq i $ and  $ ( \eta_{s1}^{ij}, \eta_{s2}^{ij}, \dots, \eta_{s{k_s}}^{ij} ) \neq 0 $ for all $ i \in \{ 1,2, \dots , n\} $.
  	Therefore, if $[T]_{B_r}^i  = \big( \eta_{i\jmath}^{rz} \big)_{ 1 \leq \jmath \leq k_i} ^{ 1 \leq z \leq k_r}  $, $r$, $i \in \{ 1,2, \dots , n\} $. Then the matrix  $[T]_{B} $ of $T $ 
  relative to  the base $B$ has the form
  \begin{equation*}
  	[T]_{B}  = \begin{bmatrix}
  	&	[T]_{B_1}^1 &  [T]_{B_2}^1  &  [T]_{B_3}^1  & \ldots & [T]_{B_n}^1 \\
  & 	0 &  [T]_{B_2}^2  &  [T]_{B_2}^2 & \ldots &[T]_{B_n}^2  \\
  &		0 &  0  & [T]_{B_2}^3 &\ldots  & [T]_{B_n}^3  \\
  & \vdots  &  \vdots & \vdots  & \ddots & \vdots \\
  &		0 &  0  &  0   &\ldots  & [T]_{B_n}^n
  	\end{bmatrix}
  \end{equation*}
  where  $[T]_{B_r}^i  =  0  $  if $ i \nleq r$  and  $[T]_{B_r}^r  \neq  0  $ for all $ r \in \{ 1,2, \dots , n\} $. To see that each $[T]_{B_r}^i$ is invertible, we notice that $[T]_{B_r}$ is invertible, so that $0 \neq det( [T]_{B_r})$. But  $ det( [T]_{B_r}) = \prod\limits_{i} det( [T]_{B_r})^i  $ and it follows that each $[T]_{B_r}^i$ is an invertible matrix. Since $T \in \mathcal{T}$ is a weight preserving so that from Proposition \ref{w(u)=w(1)}, we have $ \tilde{w}(\eta_{rt}^{rz}) =\tilde{w}(1) $ and $  \tilde{w}( \alpha \eta_{rt}^{rz}) =\tilde{w}(\alpha) $ for all $r \in [n]  $ and $\alpha \in \mathbb{F}_{q}^ {k_r}$.
  \end{proof}
	\begin{remark}
		Let $I$ and $J$ be two ideals of $P =	([n], \preceq )$.  If $I \subseteq J$ then $I \setminus M_I \subseteq J \setminus M_J $
	\end{remark}
	\begin{proposition}\label{Iso wt ej less wt ei}
		Let $v_j  \neq 0$ be the $j^{th} $ label of $ T(\beta_i e_{i,z})$ and  $T \in LIsom_{(P,w,\pi)}(\mathbb{F}_{q}^N)$. If  $\alpha_{j} \in \mathbb{F}_{q}^{k_j}$ such that $  \tilde{w}( \alpha _{j})  \leq \tilde{w}( v_j ) $  then  $ 	w_{(P,w,\pi)} {(  \alpha_{j} e_{j,z})} \leq 	w_{(P,w,\pi)} {( \beta_i e_{i,z})}  $.
	\end{proposition}
	\begin{proof}
  Since $ I_{v_{jz} e_{j,z}}^{P,\pi} \subseteq I_{T( \beta e_{i, z} )}^{P,\pi}$ so that $ I_{v_{jz} e_{j,z}}^{P,\pi} \setminus M_{v_{jz} e_{j,z}}^{P,\pi} \subseteq I_{T( \beta e_{i, z} )}^{P,\pi} \setminus M_{T( \beta e_{i, z} )}^{P,\pi} $. Thus, 
		  \begin{align*}
		  	  	w_{(P,w,\pi)} {(  \alpha_{j} e_{j,z})} 	&= \tilde{w}(\alpha_{j}) +  \sum\limits_{  k \in I_{\alpha_{j}  e_{j,z}}^{P,\pi} \setminus {M_{\alpha_{j}  e_{j,z}}^{P,\pi} }} M_w \\
		  	  &	\leq  \tilde{w}(v_{j}) +  \sum\limits_{  k \in I_{\alpha_{j}  e_{j,z}}^{P,\pi} \setminus {M_{\alpha_{j}  e_{j,z}}^{P,\pi} }} M_w \\
		  	 &\leq  \sum\limits_{  k \in M_{T(\beta_{i}  e_{i,z}) }^{P,\pi}} \tilde{w}(v_{k}) +  \sum\limits_{  k \in I_{T(\beta_{i}  e_{i,z})}^{P,\pi} \setminus {M_{T(\beta_{i}  e_{i,z})}^{P,\pi} }} M_w \\
		  	  	&= 	w_{(P,w,\pi)} {( \beta_i e_{i,z})}
		  \end{align*} 
	\end{proof}
\begin{proposition}\label{isom remark}
	If $ 	w_{(P,w,\pi)} {(  \alpha_i e_{i,z})} =	w_{(P,w,\pi)} {( \beta_j e_{j,z})}  $, then $ \tilde{w}(\alpha_i)  = \tilde{w}(\beta_j) $.	
\end{proposition}
\begin{proof}
	 For $ 0 \neq \alpha \in  \mathbb{F}_{q}^{k_i}$ and $  0 \neq \beta \in  \mathbb{F}_{q}^{k_\jmath}$. Then
	\begin{equation*}
		\begin{split}
			\tilde{w}(\alpha_i) + \sum\limits_{  k \in I_{\alpha_i  e_{i,z}}^{P,\pi} \setminus \{i \}} M_w	
			&=  \tilde{w}(\beta_j) + \sum\limits_{  k \in I_{\beta  e_{j,z}}^{P,\pi} \setminus \{j \}} M_w	\\
			\tilde{w}(\alpha_i) - \tilde{w}(\beta_j)  
			&=  \sum\limits_{  k \in I_{\beta_j  e_{j,z}}^{P,\pi} \setminus \{j \}} M_w -  \sum\limits_{  k \in I_{\alpha_i  e_{i,z}}^{P,\pi} \setminus \{i \}} M_w	
			\\
			&= t M_w  ~~~~~~~~\text{(for some integer }t)
		\end{split}
	\end{equation*}
	 Since $ 0 <  \tilde{w}(\alpha) \leq M_w$ and  $ 0 <  \tilde{w}(\beta) \leq M_w$, thus $| \tilde{w}(\alpha) - \tilde{w}(\beta) | < M_w $. So $t$ must be zero. Hence  $ \tilde{w}(\alpha)  = \tilde{w}(\beta) $.
\end{proof}
	\begin{proposition}\label{w(u)=w(1)}
		Let $T \in LIsom_{(P,w,\pi)}(\mathbb{F}_{q}^N)$ and  $\alpha_{\jmath}$ be $\jmath ^{th}$  label of $ {T(e_{i,z})} $ If  $\jmath$ is the maximal element in
		$ I_{T(e_{i,z})} ^{P,\pi}$. Then $ \tilde{w}(u_\jmath)  = \tilde{w}(1) $.
	\end{proposition}
\begin{proof}
	Since $ w_{(P,w,\pi)} {( e_{i,z} )} = 	w_{(P,w,\pi)} {( T( e_{i,z} ) )} =	w_{(P,w,\pi)} {(  \alpha_{j} e_{j,z})} $, it follows that $ \tilde{w}(u_\jmath)  = \tilde{w}(1) $.
\end{proof}
\section{Group of Linear isometries}
Considering the two subgroups $\mathcal{I}mg(\Gamma)$ and $\mathcal{T}$ constructed in the previous section, we aim to describe the group of linear isometries of  $(\mathbb{F}_{q}^N)$.  An ideal $I$ of a poset $ P $ is said to be a prime ideal if it contains a unique maximal element. 
\begin{lemma}\label{Isom prime ideal}
	If  $T \in {LIsom}_{(P,w,\pi)} (\mathbb{F}_{q}^N)$ and $0 \neq \alpha_{iz} \in  \mathbb{F}_{q}^{k_i} $ then $\langle  supp_{\pi}  (T( \alpha_{iz} e_{i,z})) \rangle $ is a prime ideal for every $i \in  \{ 1,2,\ldots,n\} $.
\end{lemma}
 \begin{proof}
 	Let  $0 \neq \alpha_{iz} \in  \mathbb{F}_{q}^{k_i} $ and $\tilde{w}(\beta) = min\{ \tilde{w}( \alpha_{iz} )  :  \alpha_{iz} \in \mathbb{F}_{q}^{k_i} \}$. We will first show that there is an element  $  j
 	\in \langle supp_{\pi}  ( T( \beta  e_{i,z} ) ) \rangle $  such that  
 	\begin{equation*}
 		w_{(P,w,\pi)} ( v_j  e_{j,z}) = w_{(P,w,\pi)} ( \beta e_{i,z} ) 
 	\end{equation*} 
 \sloppy{where $v_j $ is the $j^{th} $ label of $  T( \beta e_{i,z} ) $.  Assume that $
 	w_{(P,w,\pi)} ( v_j  e_{j,z}) <  w_{(P,w,\pi)} ( \beta e_{i,z} ) $ 
 	for every label  $v_{j} \neq 0  $ of $  T( \beta e_{i,z} )  $. If $  supp_{\pi}  ( T( \beta e_{i,z} ) ) = \{i_1,i_2,\ldots,i_s \}  $.  Then} 
 	\begin{equation*}
 		T ( \beta e_{i,z} ) =  v_{i_1} e_{{i_1},z} +  v_{i_2} e_{{i_2},z} + \ldots +  v_{i_s} e_{{i_s},z} 
 	\end{equation*}
 	where $ v_{ i_t} \in  \mathbb{F}_{q}^{k_{i_t}} $ for $ t \in  \{ 1,2,\ldots,s\} $ and,
 	by assumption, $ 	w_{(P,w,\pi)} ( v_{i_t} e_{i_t,z} ) <  w_{(P,w,\pi)} ( \beta e_{i,z}  )  $  for $ t \in  \{ 1,2,\ldots,s\} $. It follows from the linearity of $ T^{-1} $ that 
 	\begin{equation*}
 		\{i \} = supp_{\pi} ( \beta e_{i,z} ) \subseteq  \bigcup\limits_{t =1}^{s} supp_{\pi} ( T^{-1} ( v_{i_t} e_{i_t,z} ) )
 	\end{equation*}
 	which implies that $  i \in  supp_{\pi} ( T^{-1}( v_{i_t} e_{i_t,z} ) ) $  for some $ t \in  \{ 1,2,\ldots,s\} $.
 	Thus, from Proposition \ref{Iso wt ej less wt ei} ensure that if $u_i$ is the $i^{th}$ label of $  ( T^{-1}( v_{i_t} e_{i_t,z} ) )$,
 	\begin{equation*}
 		w_{(P,w,\pi)}  ( u_{i} e_{i,z} ) \leq  w_{(P,w,\pi)} ( v_{i_t} e_{i_t,z }) <  w_{(P,w,\pi)} ( \beta e_{i,z} ) 
 	\end{equation*}
 	that is,  $ \tilde{w} ( u _{i}) < \tilde{w} ( \beta ) =  min\{ \tilde{w}( \alpha_{iz} )  :  \alpha_{iz} \in \mathbb{F}_{q}^{k_i}  \}$, a contradiction. Hence, there is an element 
 	$  j \in \langle  supp_{\pi} ( T( \beta e_{i,z}) ) \rangle $  such that 
 	$	w_{(P,w,\pi)} ( v_j  e_{j,z}) = w_{(P,w,\pi)} ( \beta e_{i,z} ) 
 	$. 
 	\par  By the $(P,w, \pi)$-weight  preservation of T, 
 	\begin{equation*}
 		\begin{split}
 			\tilde{w}( v_{j} ) +  \sum\limits_{i \in {I_{ v_j e_{j,z}}^{P,\pi}} \setminus{ M_{v_j e_{j,z}}^{P,\pi} } } M_w &=	w_{(P,w,\pi)} ( v_j e_{j,z} ) \\ 
 			&= 	w_{(P,w,\pi)} ( T {(  v_j e_{j,z} )} )  \\
 			&=  \sum\limits_{i \in { M_{T {( \beta e_{j,z} )} }^{P,\pi} } } \tilde{w}( v_{i} ) +  \sum\limits_{i \in {I_{ T {( \beta  e_{j,z} )} }^{P,\pi}} \setminus{ M_{T {( \beta e_{j,z} )} }^{P,\pi} } } M_w 
 		\end{split}
 	\end{equation*}
 	such an element $j$ is unique and so $ I_{ T {( \beta e_{i,z})} }^{P,\pi} $  is a prime ideal. Now, considering any zero $ \alpha_{iz} \in \mathbb{F}_{q}^{k_i}  $, since 
 	$ supp_{\pi} ( T( \beta e_{i,z}) )  = supp_{\pi} (  \beta T( e_{i,z} ) ) =  supp_{\pi} (  T( e_{i,z} ) ) = supp_{\pi} (  \alpha_{iz} T( e_{i,z} ) ) = supp_{\pi} ( T( \alpha_{iz} e_{i,z} ) ) $ the result follows.
 \end{proof}
 \begin{lemma}\label{Iso supp e_ subset e_t}
 	\sloppy{If  $T \in LIsom_{(P,w,\pi)}(\mathbb{F}_{q}^N)$ and $ i \preceq t $, then
 	$ \langle  supp_{\pi} (T(  e_{i,z})) \rangle \subseteq  \langle  supp_{\pi}  (T(  e_{t,z})) \rangle $.}
 \end{lemma}
 \begin{proof}
 	If $i=t$, then there is nothing to prove. Let $ i \neq t $, from  Lemma \ref{Isom prime ideal}, $ \langle  supp_{\pi}  (T(  e_{i,z}))  \rangle $ and $ \langle  supp_{\pi}  (T(  e_{t,z})) \rangle $ are a prime ideals.  So there are elements $ k $ and $ j $ such that 
 $  \langle  k \rangle = \langle  supp_{\pi}  (T(  e_{i,z}))  \rangle $ and 
 $ \langle  j \rangle =  \langle  supp_{\pi}  (T(  e_{t,z})) \rangle $. If $k = j$ then we are done, so assume $k \neq j$. Thus, either 
 $ 	k \in \langle  supp_{\pi}  (T(  e_{i,z}) - T(  e_{t,z})) \rangle $ or  $
 j \in \langle  supp_{\pi}  (T(  e_{i,z}) - T(  e_{t,z})) \rangle $.
 Therefore, we have three cases to consider: \newline
 (1) If $ k \notin  supp_{\pi}  (T(  e_{i,z}) - T(  e_{t,z}))  $:  In this case, 
 $ 	k \in  supp_{\pi}  ( T(  e_{t,z}))  $ because $ k \in  supp_{\pi}  ( T(  e_{i,z}))  $. It follows that $  \langle  supp_{\pi}  (T(  e_{i,z}))  \rangle = \langle  k \rangle \subseteq \langle  supp_{\pi} (T(  e_{t,z})) \rangle  $.
 \\
 (2) If $ j \notin supp_{\pi}  (T(  e_{i,z}) - T(  e_{t,z})) $:  In this case, 
 $ 	j \in  supp_{\pi}  ( T(  e_{i,z}))  $  so  $ j <  k $.  Hence,  $  \langle  supp_{\pi}  (T(  e_{t,z}))  \rangle = \langle  j \rangle  \subsetneq \langle  k \rangle =  \langle  supp_{\pi}  (T(  e_{i,z}))  \rangle$. So,
  \begin{equation*}
  	\begin{split}
  		w_{(P,w,\pi)} ( e_{t,z} )  
  		&= 	w_{(P,w,\pi)} ( T {(  e_{t,z})} )  \\
  		&=  \tilde{w}( 1 ) +  \sum\limits_{j \in {I_{ T {(  e_{t,z})} }^{P,\pi}} \setminus{ \{ j \}} } M_w \\
  			&<  \tilde{w}( 1 ) +  \sum\limits_{j \in {I_{ T {(  e_{i,z})} }^{P,\pi}} \setminus{ \{k\}} } M_w  \\
  				&= 	w_{(P,w,\pi)} ( T {(  e_{i,z})} )
  	\end{split}
  \end{equation*}
 The second and third equality follow from Proposition \ref{w(u)=w(1)}. However, the hypothesis $ i \preceq_P t$ implies $ w_{(P,w,\pi)} ( T {(  e_{i,z})} ) \leq 	w_{(P,w,\pi)} ( e_{t,z} ) $, a contradiction.
  \\
  (3) If $ 	k, ~j \in supp_{\pi}  (T(  e_{i,z}) - T(  e_{t,z}))  $:  
  Let $x_m$  and $v_m $  be the  $ m^{th} $ labels of $T(  e_{i,z}))  $ and $T(  e_{t,z}))  $  respectively. If $u_k$  and  $u_j$ are the respectively $ k^{th} $ and  $ j^{th} $  labels of $ T(  e_{i,z}) -  T(  e_{t,z}) $, 
 \begin{equation*}
  	\begin{split}
  		w_{(P,w,\pi)} ( u_k e_{k,z} -  u_j e_{j,z}  )  
  		&\leq	w_{(P,w,\pi)} ( T {(  e_{i,z})}  - T {(  e_{t,z})} )  \\
  		&=  w_{(P,w,\pi)} ( T {(  e_{i,z}-  e_{t,z})} ) \\
  		&= 	w_{(P,w,\pi)} {(  e_{i,z} -  e_{t,z})} 
  	\end{split}
  \end{equation*}
  By hypothesis  $ i \preceq_P t$ so $ w_{(P,w,\pi)} {(  e_{i,z} -  e_{t,z})} \leq	w_{(P,w,\pi)} {( e_{t,z})}  $. And, 
  \begin{equation}\label{eqation1}
  	\begin{split}
  		w_{(P,w,\pi)} ( u_k e_{k,z} -  u_j e_{j,z}  )  
  		&\leq	w_{(P,w,\pi)} {( e_{t,z})} \\
  		&=  w_{(P,w,\pi)} ( T {( e_{t,z})} ) \\
  		&= 	w_{(P,w,\pi)} {(  v_j e_{j,z})} 
  	\end{split}
  \end{equation}
 If $x_j $  and $ v_k $ are both non-zero, then  $ j \preceq_P k$ and  $ k \preceq_P j $, a contradiction with $k \neq j$. So either $x_l $ are zero or $ v_k $ are zero. If  $x_j = 0 $ 
  then  $u_j = - v_j $, from (\ref{eqation1}) we have that $ k \preceq_P j $.  If  $ v_k = 0 $ 
  then $ u_k e_{k,z} -  u_j e_{j,z} = x_k e_{k,z} -  u_j e_{j,z}  $, and in this case, if 
  $ k \not\preceq_P j $ or $ j \prec_P k $, as $ \tilde{w}(x_k) = \tilde{w}(1) = \tilde{w}(v_j)  $ (Proposition \ref{w(u)=w(1)}), it follows $ 	w_{(P,w,\pi)} ( x_k e_{k,z} -  u_j e_{j,z}  ) > 	w_{(P,w,\pi)} (  v_j e_{j,z}  ) $, a contradiction
  with (1). Therefore  $ k \preceq_P j $.  In both cases, we have that  $ k \preceq_P j $. Hence 	$ \langle supp_{\pi}  (T(  e_{i,z})) \rangle \subseteq  \langle  supp_{\pi}  (T(  e_{t,z})) \rangle $.
 \end{proof}
\begin{proposition}\label{Isom T }
	If  $T \in LIsom_{(P,w,\pi)}(\mathbb{F}_{q}^N)$ and $0 \neq \alpha \in  \mathbb{F}_{q}^{k_i} $ then for each $i \in [n]$ there is a  $t \in [n]$, 
	\begin{equation*}
		T (\alpha_{iz} e_{i,z}) = \beta_{tz}  e_{ t,z } + u^{t}
	\end{equation*}
	where $  u^{t} \in (\mathbb{F}_{q}^N)_{{\langle t \rangle} ^{*}} $ and $ \tilde{w}(\beta_{tz} ) =  \tilde{w} ( \alpha_{iz}  )$.  In particular, if 
	$\alpha_{tz} = 1$ then $ \tilde{w}(\beta_{tz} ) =  \tilde{w} ( 1 )$ and $ \tilde{w}(\delta_{tz} \beta_{tz} ) =  \tilde{w} ( \delta_{tz})$ for all  $ \delta_{tz} \in \mathbb{F}_{q}^{k_t} $.
\end{proposition}
\begin{proof}
	\sloppy{There exist a unique $ t \in [n] $ from Lemma \ref{Isom prime ideal} such that $ \langle t  \rangle = \langle  supp_{\pi}  (T( e_{i,z})) \rangle = \langle supp_{\pi} (T( \alpha_{iz} e_{i,z})) \rangle $  and so  $ T( \alpha_{iz} e_{i,z}) \in (\mathbb{F}_{q}^N)_{{\langle t \rangle} }$. So that we get $  T (\alpha_{iz} e_{i,z}) = \beta_{tz} e_{ t,z } + u^{t} $  for some $\beta_{tz} \in \mathbb{F}_{q}^{k_t} $ and  $  u^{t} \in (\mathbb{F}_{q}^N)_{{\langle t \rangle} ^{*}} $. 
	Since $	w_{(P,w,\pi)} (T (\alpha_{iz} e_{i,z})) = 	w_{(P,w,\pi)} ( \beta_{tz} e_{t,z} )$ and $ T $ preserves weights, we have that  $	w_{(P,w,\pi)} (\alpha_{iz} e_{i,z}) = w_{(P,w,\pi)} ( \beta_{tz} e_{t,z})$. From Proposition \ref{isom remark}, we conclude that $\tilde{w}( \beta_{tz}  ) =  \tilde{w} ( \alpha_{iz}) $.}
\end{proof}
\begin{proposition}\label{T(Vi) is subset of V{ideal j}}
	If  $T \in LIsom_{(P,w,\pi)}(\mathbb{F}_{q}^N)$ for each $i \in [n]$ there is a unique $t \in [n]$, such that  $	w_{(P,w,\pi)} (T ( e_{i,z})) = 	w_{(P,w,\pi)} ( e_{t,z} )$ and 
$T(\mathbb{F}_{q}^N)_{\langle {i} \rangle} \subseteq (\mathbb{F}_{q}^N)_{{\langle {j} \rangle}} $.
\end{proposition}
\begin{proof}
	The proof follows from the Lemma \ref{Isom prime ideal} and Proposition  \ref{Isom T }. 	
\end{proof}
\begin{theorem}\label{phiT preserves label}
	Let $T : \mathbb{F}_{q}^N \rightarrow \mathbb{F}_{q}^N $ be an automorphism of  $(\mathbb{F}_{q}^N, ~d_{P,w,\pi})$, let $i \in P$ and let $j $ be the unique element of $P$ determined by $T(\mathbb{F}_{q}^N)_{i} \subseteq (\mathbb{F}_{q}^N)_{{\langle {j} \rangle}} $ and $	w_{(P,w,\pi)} (T (\alpha_{iz} e_{i,z})) = 	w_{(P,w,\pi)} ( \beta_{jz} e_{j,z} )$. Then $ dim((\mathbb{F}_{q}^N)_{i}) = dim (  (\mathbb{F}_{q}^N)_{ j })$.
\end{theorem}
\begin{theorem}\label{structureTheorem}
		If  $T \in LIsom_{(P,w,\pi)}(\mathbb{F}_{q}^N)$ and $ \alpha_{iz} \in  \mathbb{F}_{q}^{k_i} $ such that $ \tilde{w} ( \alpha_{iz}) = M_w $. Consider the map $ \phi_{T} : [n] \rightarrow [n] $  given by
	\begin{equation*}
	  \phi_{T} (i) = Max \langle  supp_{\pi}  (T( \alpha_{iz} e_{i,z})) \rangle 
	\end{equation*}
  Then: 
  \begin{enumerate}[label*=(\roman*)]
 \item  $ \phi_{T} $ is an automorphism of the labelled  poset $ (P, \pi)$.
 \item The map  $ \Phi_{T} : LIsom_{(P,w,\pi)}(\mathbb{F}_{q}^N) \rightarrow \mathcal{AUT}(P,\pi)  $ given by  $ {T}  \rightarrow \phi_{T}  $ is a surjective group homomorphism from  $ LIsom_{(P,w,\pi)}(\mathbb{F}_{q}^N)$ onto $\mathcal{AUT}(P,\pi) $ with kernel equal to $\mathcal{T}$.  In particular,  $\mathcal{T}$ is a normal subgroup of $LIsom_{(P,w,\pi)}(\mathbb{F}_{q}^N)$.
 \item The map  $ \Gamma : \mathcal{AUT}(P,\pi) \rightarrow  LIsom_{(P,w,\pi)}(\mathbb{F}_{q}^N)$ given by 
 $  \Gamma_{ \psi} = T_{ \psi }$  satisfies $ { \Phi} \circ \Gamma( { \psi } ) = { \psi } $ for all  $ \psi \in \mathcal{AUT}(P,\pi) $.  
 \end{enumerate} 
\end{theorem}
\begin{proof}
	The map  $ \phi_{T} $ is well-defined by Lemma \ref{Isom prime ideal}.
Furthermore, Lemma \ref{Iso supp e_ subset e_t} ensures that $ \phi_{T} $ is an order-preserving	map.	We claim that $ \phi_{T} $ is one-to-one. In fact, let us suppose that
$ j = \phi_{T}(i) = \phi_{T}(t) $. Since $  \phi_{T} (i) = Max \langle  supp_{\pi}  (T( \alpha_{iz} e_{i,z})) \rangle $  and $  \phi_{T} (t) = Max \langle  supp_{\pi}  (T( \alpha_{iz} e_{t,z})) \rangle $, it follows that, $ \langle  supp_{\pi}  (T( \alpha_{iz} e_{i,z})) \rangle =  \langle  j \rangle =  \langle  supp_{\pi}  (T( \alpha_{iz} e_{t,z})) \rangle$. 
\par By the $(P,w,\pi)$-weight preservation and the linearity of $ T $,
$	w_{(P,w,\pi)} (\alpha_{iz} e_{i,z} + \alpha_{iz} e_{t,z} ) = w_{(P,w,\pi)} ( T ( \alpha_{iz} e_{i,z} + \alpha_{iz} e_{t,z}) ) =  w_{(P,w,\pi)} ( T ( \alpha_{iz} e_{i,z}) + T (\alpha_{iz} e_{t,z}) ) $.
\par  Furthermore, $ \langle  supp_{\pi}  (T( \alpha_{iz} e_{i,z}) + T( \alpha_{iz} e_{t,z})) \rangle = \langle  supp_{\pi}  (T( \alpha_{iz} e_{k,z} )) \rangle$, $ k= i, t$. Hence,
\begin{equation*}
	\langle  supp_{\pi}  (T( \alpha_{iz} e_{i,z}) + T( \alpha_{iz} e_{t,z})) \rangle \subseteq \bigcup\limits_{k=i,t} \langle  supp_{\pi}  (T( \alpha_{iz} e_{k,z} )) \rangle
\end{equation*}
and both ideals on the right-hand side are assumed to be equal.
If $ u_j^i$ and $u_j^t$ are the labels of $ T( \alpha_{iz} e_{i,z}) $ and  $ T( \alpha_{iz} e_{t,z}) $ respectively, and $\beta = u_j^i  + u_j^t $ then, 
\begin{equation*}
	\langle  supp_{\pi}  (T( \alpha_{iz} e_{i,z}) + T( \alpha_{iz} e_{t,z})) \rangle =  \langle  supp_{\pi}  (T( \alpha_{iz} e_{k,z} )) \rangle; ~{k=i,t}
\end{equation*}
and since $ \tilde{w}( u_j^i) = \tilde{w}( u_j^t) = \tilde{w}(\alpha_{iz})  = M_w$  (see Proposition \ref{isom remark}),
\begin{equation*}
	\begin{split}
		w_{(P,w,\pi)} (T( \alpha_{iz} e_{i,z}) + T( \alpha_{iz} e_{t,z})) 
		& =  w_{(P,w,\pi)} ( \beta e_{j,z} ) \\ 
		& \leq w_{(P,w,\pi)} ( \alpha_{iz} e_{j,z} ) \\
		& = w_{(P,w,\pi)} (T( \alpha_{iz} e_{k,z} )) ; ~{k=i,t}
	\end{split}
\end{equation*}
which implies $  w_{(P,w,\pi)} ( \alpha_{iz} e_{i,j} + \alpha_{iz} e_{t,j} ) \leq 
w_{(P,w,\pi)} ( \alpha_{iz} e_{k,j} ) ; ~{k=i,j} $. Hence  $ i \preceq_P t$ and  $ t \preceq_P i $ and so $ i = t $.
Therefore,  $ \phi_{T} $ is one-to-one. Since $ P$ is finite, it follows that  $ \phi_{T} $ is a bijection preserving order, that is, an order automorphism. Theorem \ref{phiT preserves label} shows that $\phi_T$ lies in $\mathcal{AUT}(P,\pi) $, and this takes care of the first part. 
\par  (2) - (3) Consider now $T , ~ S \in LIsom_{(P,w,\pi)}(\mathbb{F}_{q}^N)$ and  $ i \in P$. We write  $ \phi_{T}(i) = t $ and  $ \phi_{S}(t) = k  $. This means that
$  T (e_{i,j}) = \alpha_{tz} e_{ t,j } + u^{t} $ with $ \tilde{w} ( \alpha_{tz} ) = 1$ and 
$  u^{t} \in (\mathbb{F}_{q}^N)_{{\langle t \rangle} ^{*}} $  and 
$ S ( e_{t,j}) = \beta_{tz} e_{ k,j } + u^{k} $ where $ \beta_{tz}$  and $ u^{k} $ satisfy analogous conditions. Now, 
\begin{equation*}
	ST( e_{i,j} ) = S(  \alpha_{tz} e_{ t,j } + u^{t} ) = \alpha_{tz}  \beta_{tz} e_{ k,j } +  
	\alpha_{tz}  u^{k} + S(  u^{t} ) 
\end{equation*}
and, since $  w_{(P,w,\pi)} (  u^{t} ) <  
w_{(P,w,\pi)} ( \alpha_{iz} e_{t,j} ) =  w_{(P,w,\pi)} ( \alpha_{iz} e_{k,j} ) $, it follows that,
$  w_{(P,w,\pi)} (S ( u^{t}) ) < w_{(P,w,\pi)} ( e_{k,j} ) $. Since 
$ S(  (\mathbb{F}_{q}^N)_{{\langle t \rangle} })  \subseteq  (\mathbb{F}_{q}^N)_{{\langle k \rangle}}$ and  $  w_{(P,w,\pi)} (S ( u^{t}) ) < w_{(P,w,\pi)} ( e_{k,j} ) $,  it follows that  $ S(u^{t}  )  \subseteq  (\mathbb{F}_{q}^N)_{{\langle k \rangle} ^{*}  }$ and $ 	ST( e_{i,j} ) =  \alpha_{tz} \beta_{tz} e_{ k,j } +  v^k $ with  $ v^k =  \alpha_{tz}  u^{k} + S(  u^{t} ) \in (\mathbb{F}_{q}^N)_{{\langle k \rangle} ^{*} } $. Hence 
$ \phi_{ST} ( i ) =  \phi_{S} \phi_{T} ( i )$. $\Phi$ is a group homomorphism. Given  $\phi \in Aut(P) $, $\Phi (T_{\phi }) = \phi $. This proves that $\Phi$ is surjective and that $ { \Phi} \circ \Gamma( { \phi } ) = { \phi } $ for all  $ \phi \in  Aut(P) $.  
\par Finally, $\mathcal{T} \subseteq ker (\Phi)$ because by the definition of $ T(  (\mathbb{F}_{q}^N)_{ \{i \}})  \subseteq  ( \mathbb{F}_{q}^N)_{{\langle i \rangle}}$ for all $i$. This means that, if $v = \alpha_{iz}  e_{i,j}$ then $T(v) = v' + u' $ with $ v' = \beta_{iz}  e_{i,j} \in (\mathbb{F}_{q}^N)_{ \{i \}} $, $ \tilde{w}(\beta_{iz}) =  \tilde{w}(\alpha_{iz}) $ and $ u' \in  \mathbb{F}_{q}^N)_{{\langle i \rangle}^{*}} $. Hence $\mathcal{T} = ker (\Phi)$. This shows also that $\mathcal{T}$ is a normal subgroup of  $LIsom_{(P,w,\pi)}(\mathbb{F}_{q}^N)$.
\end{proof}
 Let $M_{r \times t} ( \mathbb{F}_{q} ) = \bigg(\eta_{it}^{jz} \bigg)_{1 \leq t \leq k_i} ^{1 \leq z \leq k_j} $ be the set of all $r \times t $ matrices over $ \mathbb{F}_{q} $ and, we define $	\mathcal{U}(P,w,\pi)$ as
 \begin{equation}\label{eq2}
 	\mathcal{U}(P,w,\pi)=\left\{ (A_{ij} ) \in M_{N \times N} ( \mathbb{F}_{q} )  : 
  \begin{array}{ll}
  A_{ij}  \in M_{k_i \times k_j} ( \mathbb{F}_{q} ) \\
  A_{ij} = 0  \text{~if} ~ i \neq j \\
  A_{ii} \text{~is invertible with } \tilde{w}(\eta_{rt}^{rz}) =\tilde{w}(1) \text{ and }  \\ \tilde{w}( \alpha \eta_{rt}^{rz}) =\tilde{w}(\alpha)  \text{for all } r \in [n] \text{ and }  \alpha \in \mathbb{F}_{q}^ {k_r}
 \end{array}
  \right\}
 \end{equation}
We have a structure Theorem \ref{structureTheorem} for  	$  LIsom_{(P,w,\pi)}(\mathbb{F}_{q}^N)$, 
$\mathcal{T} $ is the group of the isometries satisfying the hypothesis of Proposition \ref{T collection}, and the 
$  \mathcal{I}mg(\Gamma)$ is the group of isometries of the form $T_\psi$  with $\psi \in \mathcal{AUT}(P, \pi) $.
\begin{theorem}\label{TImag}
	Every Linear isometry $S$ can be written in a unique way as a product of $S=F \circ T_\psi$ where $F \in \mathcal{T}$ and $T_\psi \in \mathcal{I}mg(\Gamma)$. Furthermore,
	$  LIsom_{(P,w,\pi)}(\mathbb{F}_{q}^N) \cong \mathcal{T} \rtimes  \mathcal{I}mg(\Gamma) \cong  \mathcal{U}(P,w,\pi)  \rtimes \mathcal{AUT}(P,\pi) $,
	where $ \mathcal{T} \rtimes  \mathcal{I}mg(\Gamma)$ is the semi-direct product of $\mathcal{T} $ by $  \mathcal{I}mg(\Gamma)$ induced by the action of $  \mathcal{I}mg(\Gamma)$ on $\mathcal{T}$ by conjugation and $ \cong$ denotes the group isomorphism.
\end{theorem}
\begin{proof}
	Given $S \in LIsom_{(P,w,\pi)}(\mathbb{F}_{q}^N)$, if $\psi = \psi_S $, then $F = S \circ (T_\psi)^{-1}= S \circ T_{{\psi}^{-1}}$ is in $\mathcal{T}$ and $S = ( S \circ T_{{\psi}^{-1}}) \circ T_\psi$. This expression shows that $ LIsom_{(P,w,\pi)}(\mathbb{F}_{q}^N) = \mathcal{T} \circ \mathcal{I}mg(\Gamma)$. We have
seen that $ {\Phi} \circ \Gamma( { \psi } ) = { \psi } $ for all  $ \psi \in \mathcal{AUT}(P,\pi) $ and that $\Phi(T)$ is an identity map, for all $T \in \mathcal{T}$. Since $\mathcal{I}mg(\Gamma) = \Gamma (\mathcal{AUT}(P,\pi) ) $, it follows that $ \mathcal{I}mg(\Gamma) \cap \mathcal{T} = \{Id\}$ where \textit{Id} is the identity map;
from this and from the fact that $\mathcal{T}$ is a normal subgroup of
$LIsom_{(P,w,\pi)}(\mathbb{F}_{q}^N)$ we have the first isomorphism. The second
one follows from the isomorphisms $ \mathcal{I}mg(\Gamma) \equiv \mathcal{AUT}(P,\pi) $
and $ \mathcal{T} \equiv 	\mathcal{U}(P,w,\pi) $.
\end{proof}
 \begin{corollary}
 	 $LIsom_{(P,w,\pi)}(\mathbb{F}_{q}^N) = LIsom_{(P,w_H,\pi)}(\mathbb{F}_{q}^N) $
 	if and only if $w=\alpha w_H$ for some non-negative integer $\alpha$.
 \end{corollary}
  \begin{proof}
  	\sloppy{If $w=\alpha w_H$ for some non-negative integer $\alpha \in \mathbb{F}_q$, we have that $LIsom_{(P,\alpha w_H,\pi)}(\mathbb{F}_{q}^N) = LIsom_{(P,w_H,\pi)}(\mathbb{F}_{q}^N) $.  Now	if $LIsom_{(P,w,\pi)}(\mathbb{F}_{q}^N) = LIsom_{(P,w_H,\pi)}(\mathbb{F}_{q}^N) $, since 
  	$\mathcal{U}(P,w_H,\pi) = \mathcal{U}(P,\alpha w_H,\pi)$ and 	$\mathcal{U}(P,w,\pi) = \mathcal{U}(P,\alpha w_H,\pi)$, then $w = \alpha w_H $ where $w(\alpha) = w(1)$.}
  \end{proof} 
 \subsection{Examples: Linear Isometries on $(P,w)$-space and $(P,\pi)$-space}
  The  $(P,w,\pi)$-space becomes the $(P,w)$-space (as in  \cite{wcps}) if $k_i = 1 $ for every $i\in [n]$ and the  $(P,w,\pi)$-space becomes the $(P,\pi)$-space  (as in \cite{Ebc}) if $w$ is the Hamming weight  on $\mathbb{F}_q$. Linear isometries of $(P,w)$-space and $(P,\pi)$-space is already described in \cite{wcps} and \cite{Ebc} respectively.
  With the help of the particular Theorem  \ref{TImag}, we will re-obtain linear isometries for those spaces. 
  \par \sloppy{In the case that  $k_i = 1 $ for every $i\in [n]$, $A_{ij}  \in \mathbb{F}_{q} $ from equation \ref{eq2}, we get $ \mathcal{U}(P,w,\pi)= \{ (A_{ij} ) \in M_{n \times n} ( \mathbb{F}_{q} )  : 
  A_{ij} = 0  \text{~if} ~ i \nleq j \text{ and }	w(A_{ii}) = w(1) \text{ such that } w(\alpha A_{ii}) = w(\alpha)~ \forall~ \alpha \in  \mathbb{F}_{q}  
  \} = \mathcal{U}(P,w) $ 
  and $\mathcal{AUT}(P,\pi) = \mathcal{AUT}(P) $. Then, the characterization of $ LIsom_{(P,w,\pi)}(\mathbb{F}_{q}^N)$ given in \cite{wcps} follows from the Theorem \ref{TImag} as:}
  \begin{align*}
  	LIsom_{(P,w,\pi)}(\mathbb{F}_{q}^N) \cong  \mathcal{U}(P,w)  \rtimes \mathcal{AUT}(P).
  \end{align*}
  \par Now, we consider the case when $w$ is the Hamming weight on $\mathbb{F}_q$,  $(P,w,\pi)$-space is then $(P,\pi)$-space.  Thus, from equation \ref{eq2} we get:
  \begin{equation}
 	\mathcal{U}(P,w,\pi)=\left\{ (A_{ij} ) \in M_{N \times N} ( \mathbb{F}_{q} )  : 
 	\begin{array}{ll}
 		A_{ij}  \in M_{k_i \times k_j} ( \mathbb{F}_{q} ) \\
 		A_{ij} = 0  \text{~if} ~ i \neq j \\
 		A_{ii} \text{~is invertible}
 	\end{array}
 	\right\}
 \end{equation}
Then, the characterization of $ LIsom_{(P,w,\pi)}(\mathbb{F}_{q}^N)$ given in \cite{Ebc} follows from the Theorem \ref{TImag} as:
\begin{align*}
	LIsom_{(P,w_H,\pi)}(\mathbb{F}_{q}^N) \cong  \mathcal{U}(P,w_H,\pi)  \rtimes \mathcal{AUT}(P,\pi).
\end{align*}
 \par  We now consider the case when $P$ is an antichain. The $\pi$-weight of $x=x_1+x_2+\dots+x_n \in \mathbb{F}_{q}^N$ is defined to be 
 \begin{equation*}
 	w_\pi (x) = | \{ i : x_i \neq 0 \}|
 \end{equation*}
 and the $(P,\pi)$-weight of $x$ is $	w_{(P,\pi)} (x) = 	w_\pi (x) $.
  In this case $\langle i \rangle = \{ i \}$ for each $ i \in [n]$, and hence the upper-triangular maps $T$ take $\mathbb{F}_{q}$ isomorphically onto itself. Therefore, 
\begin{equation*}
	\mathcal{T} \cong 	LIsom(k_1, \tilde{w}, \mathbb{F}_{q}) \times  LIsom(k_2, \tilde{w}, \mathbb{F}_{q}) \times \cdots \times LIsom(k_n, \tilde{w}, \mathbb{F}_{q}) 
\end{equation*}
  where $LIsom(\tilde{w}, \mathbb{F}_{q}) $ is the group of the linear transformation $T : \mathbb{F}_{q} \rightarrow \mathbb{F}_{q} $ that preserves the weight $\tilde{w}$. 
  \par Given $N=k_1 + k_2+ \ldots + k_n$, let
  $t_1, t_2 , \ldots, t_l$ be the $l$ distinct elements ($t_1 > t_2 > \ldots > t_l > 0$) in the parts  $k_1,k_2,\ldots,  k_n$ with multiplicity $r_1,r_2,\ldots,r_l$ respectively so that $ \sum\limits_{s=1}^{l} r_s  t_s=k_1+k_2+\dots+ k_n=N $. Let $\pi(N)=[t_1]^{r_1}  [t_2]^{r_2} \ldots [t_l]^{r_l}$ denote as a partition of $N$. 
  On the other hand $\mathcal{AUT}(P) \cong S_n$ and  $\mathcal{AUT}(P, \pi)$ can be identified with a  subgroup of $S_n$. Thus, $\mathcal{AUT}(P, \pi)$ only permutes those vertices with same labels and therefore 
\begin{equation*}
	\mathcal{AUT}(P, \pi) \cong S_{r_1} \times  S_{r_2} \times \ldots \times  S_{r_l}.
\end{equation*}
 From Theorem \ref{TImag} it follows that
\begin{align*}
	LIsom_{(P,w_H,\pi)}(\mathbb{F}_{q}^N) \cong  \bigg(\prod\limits_{i=1}^{n}	LIsom(k_i, \tilde{w}, \mathbb{F}_{q}) \bigg) \rtimes  \bigg( \prod\limits_{i=1}^{l}	S_{r_i} \bigg).
\end{align*}
\bibliographystyle{amsplain}

\end{document}